\documentclass{amsart}

\usepackage{bm}

 \usepackage{upgreek}

\usepackage{amsopn, amsthm, amsgen, amscd,amsmath, amssymb}
\DeclareMathAlphabet{\mathbbm}{U}{bbm}{m}{n}
\usepackage{pifont}
\usepackage{color}
\usepackage{tikz}
  \usepackage{graphicx}
\usepackage{epstopdf}
\usepackage[small,nohug,heads = LaTeX]{diagrams} \diagramstyle[labelstyle=\scriptstyle]

\definecolor{CadetBlue}{cmyk}{0.62, 0.57, 0.23, 0 }
\definecolor{black}{cmyk}{1, 0.5, 0, 0 }
\definecolor{RedViolet}{cmyk}{0.07, 0.9, 0, 0.34 }
\definecolor{SeaGreen}{cmyk}{0.69, 0, 0.5, 0}

\DeclareMathAlphabet{\mathpzc}{OT1}{pzc}{m}{it}

\newcommand{\C}{\mathbb C}
\newcommand{\D}{\mathbb D}

\newcommand{\F}{\mathbb F}

\newcommand{\N}{\mathbb N}
\newcommand{\PR}{\mathbb P}
\newcommand{\Q}{\mathbb Q}
\newcommand{\Z}{\mathbb Z}

\newcommand{\I}{\mathbb I}

\newcommand{\gota}{\mathfrak a}
\newcommand{\gotb}{\mathfrak b}

\newcommand{\gotm}{\mathfrak m}
\newcommand{\gotn}{\mathfrak n}
\newcommand{\gotp}{\mathfrak p}
\newcommand{\gotq}{\mathfrak q}

\newcommand{\gotP}{\mathfrak P}

\newcommand{\interior}[1]{%
{\kern0pt#1}^{\mathrm{o}}%
}

\DeclareMathOperator{\ord}{ord}
\DeclareMathOperator{\Gal}{Gal}

\newtheorem{theo}{Theorem}
\newtheorem{lemm}{Lemma}
\newtheorem{prop}{Proposition}
\newtheorem{coro}{Corollary}
\newtheorem*{nntheo}{Theorem}
\newtheorem*{MainThm}{Main Theorem}

\theoremstyle{definition}

\theoremstyle{remark}

\newtheorem{note}{Note}

\newtheorem*{clai}{Claim}

\title[Modular Invariant of Rank 1 Drinfeld Modules]{Modular Invariant of Rank 1 Drinfeld Modules and Class Field Generation}
\author{L. Demangos}
\address{Xi'an Jiaotong - Liverpool University, Department of Mathematical Sciences, Mathematics Building Block B, 111 Ren'ai Road, Suzhou Dushu Lake Science
and Education Innovation District, Suzhou Industrial Park, Suzhou, Peoples Republic of China, 215123}
\email{Luca.Demangos@xjtlu.edu.cn}
\author{T.M. Gendron}
\address{Instituto de Matem\'{a}ticas -- Unidad Cuernavaca, Universidad
Nacional Aut\'{o}noma de M\'{e}xico, Av. Universidad S/N, C.P. 62210
Cuernavaca, Morelos, M\'{e}xico}
\email{tim@matcuer.unam.mx}
\date{\today}

\begin{document}

\maketitle
\date{\today}

\begin{abstract}  The modular invariant of rank 1 Drinfeld modules is introduced and used to formulate and prove an exact analog of the Weber-Fueter theorem for
global function fields. The main ingredient in the proof is a version of Shimura's Main Theorem of Complex Multiplication for global function fields, which is also proved here.
\end{abstract}
\tableofcontents

\section{Introduction}

Let $K/\F_{q}(T)$ be a finite extension associated to a morphism of curves $\Upsigma\rightarrow \PR^{1}$ defined over $\F_{q}$.
Choose a point in $\Upsigma$ lying over $\infty\in \PR^{1}$, also denoted ``$\infty$'', and let $A\subset K$ be the Dedekind ring of functions regular outside of $\infty$.  Denote by $K_{\infty}$ the completion
of $K$ at $\infty$ and by $\C_{\infty}$ the completion of an algebraic closure $\overline{K_{\infty}}$.  

Using as a template the modular invariant
of rank 2 Drinfeld modules defined by Gekeler \cite{Gekeler}, a modular invariant for rank 1 Drinfeld modules may be defined: which may be viewed, equivalently, as a function of the ideal class group
\[ j:{\sf Cl}(A)\longrightarrow K_{\infty} .\]
If $H_{A}$ is the Hilbert class field associated to $A$ (the maximal abelian unramified extension of $K$ totally split at $\infty$), in \S \ref{CFGen} the following theorem is proved.
\begin{nntheo} For any ideal class $\mathfrak{a}\in {\sf Cl}(A)$,
\[   H_{A} = K(j(\mathfrak{a})).\]
\end{nntheo}

Let $K^{\rm ab}_{A}$ be the maximal abelian extension of $K$ totally split at $\infty$.  Given $\mathfrak{m}\subset A$ a modulus, let $K^{\mathfrak{m}}_{A} = K^{\rm ab}_{A}\cap K^{\mathfrak{m}}$,
where $K^{\mathfrak{m}}$ is the usual ray class field defined by Class Field Theory.   
Let $\D$ be a rank 1 Drinfeld $A$-module, whose coefficients are assumed to belong to $H_{A}$ = the minimal field of definition of $\D$.  Let $e:\C_{\infty}/\Uplambda\rightarrow \C_{\infty}$ be the associated
exponential, where $\Uplambda = \upxi \mathfrak{a}$ for some ideal $\mathfrak{a}\subset A$ and $\upxi\in\C_{\infty}$. 
\S \ref{CFGen} also contains a proof of the following result.
  \begin{nntheo} $  K^{\mathfrak{m}}_{A} = H_{A} \big(e(\upxi m)|\; m\in \mathfrak{m}^{-1}\mathfrak{a}/\mathfrak{a}\big). $
\end{nntheo}

There are other versions of explicit class field theory for function fields, see for example \cite{Drinfeld}, \cite{Hayes2}, \cite{Hayes}, \cite{Gekeler}, \cite{AngPell}. 
The novelty in the approach presented here is two-fold: 1) in using the modular invariant, a close
analog of the classical Weber-Fueter Theorem \cite{Serre}, \cite{Si} is obtained and 2) the results give generation for the smaller non-narrow ray class fields.

The proofs do not use results from previous works on explicit class field generation. The main tool in proving the theorems in this paper is the function field version of 
Shimura's Main Theorem of Complex Multiplication \cite{Shi}: this does not seem to appear in the literature, so a proof of it is given in \S \ref{MT}.   

As a consequence of the Main Theorem, the modular invariant satisfies the equivariance
\begin{align}\label{equi}   j(\mathfrak{b}^{-1}\mathfrak{a}) = j(\mathfrak{a})^{\upsigma_{\mathfrak{b}}}  \end{align}
for any $\mathfrak{a},\mathfrak{b}\in {\sf Cl}(A)$, where $\upsigma_{\mathfrak{b}} $ is the automorphism corresponding to $\mathfrak{b}$ by Class Field Theory.  See Theorem \ref{t3} of \S \ref{CFGen}.  The latter property, together with
Theorem \ref{jadifffrom1} of \S \ref{ModInvSec},  allows
one to conclude that $j$ is injective (i.e.\ a complete invariant), see Corollary \ref{cor}.  Once it is shown that $j(\mathfrak{a})\in H_{A}$, injectivity along with (\ref{equi}) show its orbit by ${\rm Gal}(H_{A}/K)$ is of
maximal size, and thus it is a primitive generator of $H_{A}$.  The proof
of the explicit expression for  $  K^{\mathfrak{m}}_{A} $ is a straight-forward application of the Main Theorem.

\section{Basic Notation}
We refer to \cite{Goss}, \cite{Hayes},  \cite{Th} for basic ideas and results of function field arithmetic. We begin by fixing some notation once and for all.
Let $\mathbb{F}_{q}$ is the finite field with $q$ elements, $q$ a power of a prime $p$ and fix a finite extension 
\[ K/\F_{q}(T)\] 
which we assume is induced by a morphism of curves $ \Upsigma_{K}\rightarrow \PR^{1} $ defined over $\F_{q}$.
Fix a closed point $\infty \in \Upsigma_{K}$ with associated valuation $v=v_{\infty}$ and define the Dedekind domain
\[ A:=\{f\in K \text{ regular outside of $\infty$}\}. \] 
Denote by
$K_{\infty}$ the completion of $K$ with respect to $v$ and by $\F_{\infty}$ its field of constants, an extension of $\F_{q}$ of degree $d_{\infty}$.  The completion
 of an algebraic closure of $\overline{K_{\infty}}$ is denoted $\C_{\infty}$.  For $x\in \C_{\infty}$, we define the degree by $\deg (x):= -d_{\infty}v(x)$ and the absolute value
$|\cdot |= q^{\deg (\cdot )}=q^{-d_{\infty} v(\cdot )}$.

\section{Class Field Theory}\label{CFT}

Let $\I_{K}$ be the $K$-id\`{e}les
 and let  \[ K^{\rm ab}_{A}\subset K^{\rm ab}\] be the maximal abelian extension of $K$ totally split over the point  $\infty$. 
 Likewise, the Hilbert class field 
 \[ H_{A}\subset K^{\rm ab}_{A}\]
 is the maximal abelian unramified extension of $K$ totally split over $\infty$.  See \cite{Ro}.
 
  For $L/K$ with $L\subset K^{\rm ab}_{A}$, let $\mathfrak{m}\subset A$ be the product of primes that ramify in $L$ and denote by ${\sf I}(\mathfrak{m})$ the group of ideals
  relatively prime to $\mathfrak{m}$.  For each prime $\mathfrak{p}\in {\sf I}(\mathfrak{m})$, 
the Artin symbol is defined
\[  (\mathfrak{p}, L/K) :=\upsigma_{\mathfrak{p}}\in {\rm Gal}(L/K)  \]  
where $\upsigma_{\mathfrak{p}}$ is the Frobenius associated to $\mathfrak{p}$, and then one extends multiplicatively to obtain the Artin map
\[  (\cdot, L/K) :{\sf I}(\mathfrak{m})\longrightarrow   {\rm Gal}(L/K) .\]

In the id\`{e}lic language, the usual reciprocity epimorphism (see \cite{Ta}, Theorem 5.1)
\[ [\cdot ,K]: \I_{K} \longrightarrow  {\rm Gal}(K^{\rm ab}/K)    \]
 induces an $A$-reciprocity map, which appears as the lower horizontal arrow of the following commutative diagram
  \begin{diagram}
 \I_{K} & \rTo^{[\cdot, K]} &  {\rm Gal}(K^{\rm ab}/K) \\
||& &\dTo \\
  \I_{K} & \rTo^{[\cdot, K]} &  {\rm Gal}(K^{\rm ab}_{A}/K) ,\\
 \end{diagram}
the vertical arrow on the right being the canonical projection. Thus the $A$-reciprocity map is an epimorphism with kernel containing $K^{\times}$.
If we denote by $(s)\subset A$ the ideal associated to $s\in\I_{K}$ then 
\[   [s, K] |_{L} =  \left((s), L/K\right)  .\]
 If $\mathfrak{p}\subset A$ is a prime unramified in $L$, $\uppi_{\mathfrak{p}}$ is a uniformizer of $K_{\mathfrak{p}}^{\times}$, identified with the id\`{e}le having
 $1$'s at all other places,
 then
 \begin{align}\label{FrobNote}  [\uppi_{\mathfrak{p}}, K]|_{L} = \upsigma_{\mathfrak{p}} .\end{align}

Let $\mathfrak{m}\subset A$ be a modulus (an ideal) and define 
\[ K^{\mathfrak{m}}_{A}  :=K^{\mathfrak{m}}\cap L^{\rm ab}_{A}   \] to be the subfield of the usual ray class field $K^{\mathfrak{m}}$ completely split at $\infty$.  Let $U^{\mathfrak{m}}_{A}$
be the subgroup of $\I_{K}$ defined
\[ U^{\mathfrak{m}}_{A} = \{ s\in \I_{K}|\; \text{for all }\mathfrak{p}\subset A, \;s_{\mathfrak{p}}\in A_{\mathfrak{p}}^{\times} \text{ and satisfies } s_{\mathfrak{p}} \equiv 1\mod \mathfrak{m}A_{\mathfrak{p}}^{\times} \}. \]
 The reciprocity map induces an isomorphism
\begin{align}\label{RayClassChar}   [\cdot, K]: \I_{K}/K^{\times} U^{\mathfrak{m}}_{A} \longrightarrow {\rm Gal}( K^{\mathfrak{m}}_{A}/K). \end{align}
 See \S1 of \cite{Auer} as well as \cite{Hayes2}.

\section{Drinfeld-Hayes Theory}

The Frobenius automorphism acting on $\C_{\infty}$ is denoted $\uptau (x)=x^{q}$, and for any subfield $L\subset \C_{\infty}$, 
\[ L\{ \uptau \}\] is the noncommutative algebra of additive polynomials in $\uptau$ with product given by composition.
Let \[ \upiota: A\rightarrow L\]be a ring homomorphism, making $L$ an $A$-field: the characteristic is defined to be $\wp = {\rm Ker}(\upiota )$.  A rank 1 Drinfeld $A$-module defined over $L$
\[   \D = (\C_{\infty}, \uprho ), \]
consists of an $\F_{q}$-algebra homomorphism
\[ \uprho : A\longrightarrow L\{ \uptau\} \subset \C_{\infty}\{ \uptau\} \]
in which each $\uprho_{a}:= \uprho (a)$ has the form
 \[ \uprho_{a}(\uptau ) = \upiota (a)\uptau^{0} + a_{1}\uptau +\dots +a_{d}\uptau^{d},\quad d=\deg (a),\quad a_{1},\dots ,a_{d}\in L,\;\; a_{d}\not=0 .\]
Note that $a_{0}= \upiota (a)$ implies that
\[ \upiota = D\circ \uprho,\]
where $D=d/dx$ is the derivative with respect to $x$.  

For characteristic $0$ we have the following analytical version of rank 1 Drinfeld $A$-modules. By a rank 1 $A$-lattice is meant a discrete rank 1 $A$-submodule $\Uplambda$ of $\C_{\infty}$.  By Dirichlet's Unit Theorem (see \cite{Cohn}, Theorem 3.3), $A^{\times}\cong \mathbb{F}_{q}^{\times}$, which implies that
$A$ and all of its fractional ideals are rank 1 $A$-lattices in $\C_{\infty}$, and any rank 1 $A$-lattice $\Uplambda$ is of the form $\upxi\mathfrak{a}$ for $\upxi\in \C_{\infty}$ and $\mathfrak{a}$ a fractional $A$-ideal.   To  
an $A$-lattice $\Uplambda$ we may associate the exponential function
\[ e_{\Uplambda}:\C_{\infty}\longrightarrow \C_{\infty},\quad e_{\Uplambda}(z) = z\prod_{0\not= \uplambda\in\Uplambda}\left(1- \frac{z}{\uplambda}\right). \]
The exponential function in turn defines a unique Drinfeld $A$-module $\D=(\C_{\infty},\uprho)$, isomorphic as an $A$-module to $\C_{\infty}/\Uplambda$ via $e_{\Uplambda}$ i.e.\
\begin{align}\label{fundform}  e_{\Uplambda}(az) = \uprho_{a} (e_{\Uplambda}(z)),\quad \text{for all }a\in A.\end{align}
Every rank 1 Drinfeld $A$-module $\D=(\C_{\infty},\uprho )$ in characteristic 0 may be obtained in this way (see \cite{Th}, Theorem 2.4.2): we denote by $\Uplambda_{\uprho}$ and $e_{\uprho}$ the corresponding
lattice and exponential map.  It follows then that a characteristic $0$ Drinfeld $A$-module satisfies ${\rm End}(\D)\cong A$ (see \cite{Goss}, Theorem 4.7.8).

We fix a sign function: a homomorphism
\[ {\rm sgn}:K_{\infty}^{\times}\longrightarrow \F_{\infty}^{\times}\]
which is the identity on $\F_{\infty}^{\times}$.  There are exactly \[ \# \F_{\infty}^{\times}=q^{d_{\infty}}-1\] sign functions.  For any $\upsigma\in {\rm Gal}(\F_{\infty}/\F_{q})$,
a twisted sign function is a homomorphism of the form $\upsigma\circ {\rm sgn}$. 
For $\D =(\C_{\infty},\uprho )$ a rank 1 Drinfeld module, the top coefficient \[ \upmu_{\uprho}(a):= a_{d} \]
is non zero; the map $a\mapsto  \upmu_{\uprho}(a)$ may be extended to $K_{\infty}^{\times}$, and we say that $\D$ is sign normalized or a Hayes module if $\upmu_{\uprho}$ is a twisting
of the sign function.   See \cite{Hayes}.

 If $h_{A}$ is the class number of $A$ then there are exactly $h_{A}$ isomorphism classes of rank 1 Drinfeld modules over $\C_{\infty}$, and each class contains
exactly $(q^{d_{\infty}}-1)/(q-1)$ Hayes modules.  Each Hayes module $\D=(\C_{\infty},\uprho )$ is thus associated to a class of ideal $\mathfrak{a}$ which we assume integral.  Then, there is a transcendental element
$\upxi_{\uprho}\in \C_{\infty}$ so that the Drinfeld module of the lattice $\Uplambda_{\uprho}:=\upxi_{\uprho} \mathfrak{a}$ is $\D$.

Recall from \S \ref{CFT} the Hilbert class field $H_{A}$ associated to $A$. To state the results of Hayes Theory, we also require the narrow version, which is defined as follows.
The narrow class group of $A$ is the quotient
\[ {\sf Cl}_{1}(A) :={\sf I}(A)/ {\sf P}_{1}(A)\] 
where ${\sf I}(A)$ is the group of fractional ideals of $A$ and  ${\sf P}_{1}(A)$ is the subgroup of principal ideals generated by a ${\rm sgn}$ one element.  Then 
\[ h_{A}^{1}:=\# {\sf Cl}_{1}(A) =h_{A} (q^{d_{\infty}}-1)/(q-1) \]
where $h_{A}$ is the class number of $A$.  In the id\`{e}lic language (see \cite{Th}, page 80), if we let
\[ U_{A}^{1} :=\{ s\in \I_{K}|\; \text{for all }\mathfrak{p}\subset A, \;s_{\mathfrak{p}}\in A_{\mathfrak{p}}^{\times} \text{ and } {\rm sgn}(s_{\infty})=1\} ,\]
and if $\uppi_{\infty}$ is a uniformizer of $K_{\infty}$ with ${\rm sgn}(\uppi_{\infty})=1$, then the narrow Hilbert class field 
\[ H^{1}_{A}\]
is defined to be the class field corresponding to the group \[ K^{\times}\cdot \uppi_{\infty}^{\Z} \cdot U_{A}^{1}\subset\I_{K}.\]  Thus, 
 Artin reciprocity gives an isomorphism
\[  [\cdot, K]: \I_{K}/ \left( K^{\times}\cdot \uppi_{\infty}^{\Z} \cdot U_{A}^{1}\right)\longrightarrow {\rm Gal}(H_{A}^{1}/K); \]
which induces, on the level of ideal classes, an isomorphism
\[ {\sf Cl}_{1}(A) \longrightarrow {\rm Gal}(H_{A}^{1}/K) ,\quad \mathfrak{a}\longmapsto \upsigma_{\mathfrak{a}}.\]
When $d_{\infty}=1$ i.e. $\F_{\infty}=\F_{q}$
then $H_{A}^{1}=H_{A}$.

The Hilbert class fields have the following relation to rank 1 Drinfeld modules: 
\begin{itemize}
\item[i.] $H_{A}$ is the minimal field of definition
of any rank 1 Drinfeld $A$-module.   
\item[ii.] $H_{A}^{1}$ is the field generated over $K$ by the coefficients
of $\uprho_{a}$, for any Hayes module $\D=(\C_{\infty},\uprho )$ and any fixed $a\in A$.  
\end{itemize}
Thus the coefficients of any Hayes module belong to $H_{A}^{1}$.  It follows from (\ref{fundform})  that for a Hayes module, the coefficients of the exponential $e_{\uprho}$ are also in $H_{A}^{1}$.

If $\mathfrak{a}$ is an ideal and $\D=(\C_{\infty},\uprho )$ is a Hayes module, the set $\{ \uprho_{\upalpha}|\; \upalpha\in\mathfrak{a}\}$ is a principal left ideal of $H_{A}^{1}\{ \uptau\}$ and has a unique generator $\uprho_{\mathfrak{a}}$ of ${\rm sgn}$ 1.
Then there exists a unique Hayes module denoted \[ \mathfrak{a}\ast\D=(\C_{\infty},\mathfrak{a}\ast \uprho)\] for which
\[  \uprho_{\mathfrak{a}} \circ \uprho_{a} =(\mathfrak{a}\ast \uprho)_{a} \circ \uprho_{\mathfrak{a}} .  \]
That is, $\uprho_{\mathfrak{a}}$ defines an isogeny  \[ \uprho_{\mathfrak{a}}:\D\longrightarrow \mathfrak{a}\ast\D.\] 
The set of Hayes $A$-modules is a principal homogeneous space for the $\ast$ action of ${\sf Cl}_{1}(A)$,
and we have
\begin{align}\label{HayesThm}  \mathfrak{a}\ast \D = \D^{\upsigma_{\mathfrak{a}}}  \end{align}
where $\upsigma_{\mathfrak{a}}$ is the automorphism associated to the narrow class of $\mathfrak{a}$ by Class Field Theory, and where for any automorphism $\upsigma$ of $\C_{\infty}$, $\D^{\upsigma}=(\C_{\infty}, \uprho^{\upsigma})$, with $ \uprho^{\upsigma}= \upsigma\circ \uprho$.

We will need the following results concerning reduction.   Let $\gotp$ be a prime ideal of $A$. 
Let $L$ be a finite Galois extension of $K$, $\mathcal{O}_{L}\supset A$ the integral closure of $A$ in $L$ and
$\gotP\subset \mathcal{O}_{L}$ a prime ideal of $L$ which divides $\gotp$. 
For $\D = (\C_{\infty},\uprho )$ a Drinfeld module defined over $L$, we recall that $\D$ has good reduction with respect to $\mathfrak{P}$ if $\D$ is, up to isomorphism, a Drinfeld module 
with coefficients in the local ring $\mathcal{O}_{\mathfrak{P}}\subset L$ and the reduction mod $\mathfrak{P}$ is also a rank 1 Drinfeld module.  In this case, we denote by $\widetilde{\mathbb{D}}=(\C_{\infty},\widetilde{\uprho})$
the reduction of $\D$ mod $\mathfrak{P}$.

\begin{lemm}\label{1}
Let $\mathbb{D}_{1}=(\C_{\infty},\uprho_{1})$, $\mathbb{D}_{2}=(\C_{\infty},\uprho_{2})$ be two rank $\text{\rm 1}$ Drinfeld modules over $L$, both with good reduction at $\gotP$.  Then the natural reduction map:
\[ {\sf Hom}(\D_{1},\D_{2})\longrightarrow {\sf Hom}(\widetilde{\D}_{1},\widetilde{\D}_{2})\]\[\upphi \mapsto \widetilde{\upphi}\]is injective and $\deg(\upphi)=\deg(\widetilde{\upphi})$.
\end{lemm}
\begin{proof}
By definition of good reduction $\uprho_{1}$ and $\uprho_{2}$ have, up to isomorphism, coefficients in $\mathcal{O}_{\mathfrak{P}}$  and for every $a\in A$ the additive polynomials $\widetilde{\uprho}_{1,a}(\uptau )$ and $\widetilde{\uprho}_{2,a}(\uptau )$ have degree as polynomials in $\uptau$ equal to $\deg(a)$.   If we denote 
\[ \uprho_{i,A}: =\{  \uprho_{i,a},\; a\in A\} \cong A,\quad i=1,2 ,\] then
the preceding remark implies in particular that the reduction ring homomorphism
\begin{align*}\uprho_{i,A} \longrightarrow \widetilde{\uprho}_{i,A} ,\quad \uprho_{i,a}\longmapsto \widetilde{\uprho}_{i,a}, \;\; i=1,2, \end{align*} 
has trivial kernel, hence is injective and thus $\widetilde{\uprho}_{i,a}\not=0$ for $a\not=0$.  Now choose $0\not=\upphi \in {\sf Hom}(\D_{1},\D_{2})$ an isogeny. 
By \cite{Goss}, Proposition 4.7.13, there exists a dual isogeny $\widehat{\upphi}\in {\sf Hom}(\D_{2},\D_{1})$ such that for some $a\in A$ 
we have that $\widehat{\upphi}\upphi=\uprho_{1,a}$.   Note that
since
\begin{align}\label{hattildecom} \widetilde{ \widehat{\upphi}\upphi} =\widetilde{ \widehat{\upphi}} \widetilde{\upphi}= \widetilde{ \uprho}_{1,a}, \end{align}
it follows that the reduction of $\upphi$ modulo $\gotP$ cannot be $0$, otherwise $\widetilde{\uprho}_{1,a}=0$, contradicting the hypothesis of good reduction. This proves the first part of the statement. 
Now observe that $\deg(\widetilde{\upphi})\leq \deg(\upphi)$. Since $\deg(\widetilde{\uprho}_{1,a})=\deg(\uprho_{1,a})=\deg a$, 
if $\deg(\widetilde{\upphi})$ is strictly less than $ \deg(\upphi)$, (\ref{hattildecom}) implies that $\deg(\widetilde{\widehat{\upphi}})> \deg(\widehat{\upphi})$ (in order to maintain  $\deg \widetilde{\uprho}_{1,a} =\deg\uprho_{1,a}$) which is clearly impossible. 
\end{proof}

 In what follows, we assume that $\D$ is a Hayes $A$-module defined over $L$ with good reduction at $\mathfrak{P}$.  
Let $\gotm\subset A$ be a fixed modulus (integral ideal) and recall that the $\gotm$-torsion module of $\D$ is defined
\[\D[\gotm]:=\{x\in \C_{\infty},\textsl{ }\uprho_{\gotm}(x)=0\} .\]  
If $\D[\gotm]\subset L$,  then since $\D[\mathfrak{m}]$ consists of the roots of the monic polynomial $\uprho_{\mathfrak{m}}$, $\D[\mathfrak{m}]\subset\mathcal{O}_{L}\subset\mathcal{O}_{\mathfrak{P}}$: it then
makes sense to reduce $ \D[\gotm]$ modulo $\mathfrak{P}$.  Denote by
  \[  \widetilde{\D}[\gotm]\]
  the $\gotm$-torsion points of the reduced Drinfeld module $\widetilde{\D}$.

\begin{lemm}\label{injtor}  Suppose that $\D [\mathfrak{m}]\subset L$ and $\mathfrak{P}\not|\mathfrak{m}$.  Then the reduction map \[\D[\gotm]\longrightarrow \widetilde{\D}[\gotm]\] is injective. 
\end{lemm}
\begin{proof}   We note that the Lemma is true for $\mathfrak{m}=(m)$ principal: indeed, the polynomial $\uprho_{m}$ defining the $m$-torsion points
is $\mathfrak{P}$ primitive (i.e.\ $\not\equiv 0$ mod $\mathfrak{P}$) and since we have good reduction at $\mathfrak{P}$, $\widetilde{\D}$ is a rank 1 Drinfeld module hence $\widetilde{\uprho}_{m}$ reduces to a separable polynomial.  Then, we may apply Hensel's Lemma to conclude that the reduction map on $m$-torsion is injective (in fact bijective).   For general $\mathfrak{m}$, we first remark that we may choose generators $m_{1},m_{2}$ of $\mathfrak{m}$ both of which are not in $\mathfrak{P}$.  Indeed, by hypothesis, at least say $m_{1}\not\in \mathfrak{P}$: if $m_{2}\in \mathfrak{P}$ then we replace $m_{2}$ by $m_{1}+m_{2}\notin \gotP$.  We claim that $\D[\gotm]=\D[m_{1}]\cap \D[m_{2}]$. Indeed, the inclusion $\D[\gotm]\subset \D[m_{1}]\cap \D[m_{2}]$ follows from
\[\D[\gotm]=\bigcap_{m\in \gotm}\D[m]. \] On the other hand, for $m=am_{1}+bm_{2}$, $\uprho_{m}=\uprho_{a}\uprho_{m_{1}}+\uprho_{b}\uprho_{m_{2}}$ and the vanishing of both $\uprho_{m_{1}}$ and $\uprho_{m_{2}}$ implies that of $\uprho_{m}$, which proves the claim. Since the reduction map commutes with the intersection the proof is completed. 
\end{proof}


\begin{note}\label{RedFrob} Let $\D$ be a Hayes module, $\mathfrak{p}\subset A$ a prime ideal and consider the isogeny
\[ \uprho_{\mathfrak{p}}: \D\longrightarrow \mathfrak{p}\ast \D\]
between Hayes modules.  Assume that there exists a prime $\mathfrak{P}$ in $L$ above $\mathfrak{p}$ with respect to which $\D$ has good reduction.  
Then $\mathfrak{p}\ast \D$ also has good reduction at $\mathfrak{P}$ and the associated reduction map 
\[ \widetilde{\uprho}_{\mathfrak{p}}: \widetilde{\D}\longrightarrow \widetilde{\mathfrak{p}\ast \D} \]
is equal to $\uptau^{\deg(\mathfrak{p})}$.  See for example Theorem 3.3.4 on page 71 of \cite{Th} and its proof.
\end{note}


\section{Main Theorem of Complex Multiplication for Function Fields}\label{MT}

In this section we state and prove a version of Shimura's Main Theorem of Complex Multiplication \cite{Shi} in the setting of rank 1 Drinfeld modules.  


Let $\gotp\subset A$ be a prime ideal and let $K_{\gotp}$ be the completion of $K$ at $\gotp$, with $A_{\gotp}$ the corresponding completion of $A$. Similarly, if $\gota$ is a fractional ideal of $A$ let $\gota_{\gotp}:=\gota A_{\gotp}$. 

\begin{lemm}\label{Lemma 8.1}
Let $\gota$ be a fractional ideal of $A$. Then there is an isomorphism of $A$-modules:\[K/\gota\simeq \bigoplus_{\gotp}K_{\gotp}/\gota_{\gotp}.\]
\end{lemm}
\begin{proof}
See \cite{Si}, Remark 8.1.1, where we note additionally that the isomorphism respects the $A$-module structure.
\end{proof}

For every $x=(x_{v})\in \mathbb{I}_{K}$ the ideal generated by $x$ is defined \[(x):=\prod_{\gotp\subset A}\gotp^{\ord_{\gotp}(x_{\gotp})}.\]
Moreover, for every $\gota\subset A$ an ideal we denote \[x\gota:=(x)\gota.\]
Note that \[(x\gota)_{\gotp}=(x)\gota A_{\gotp}=x_{\gotp}\gota A_{\gotp}=x_{\gotp}\gota_{\gotp}.\]
Lemma \ref{Lemma 8.1} now gives the following isomorphism of $A$-modules :\[K/x\gota\simeq \bigoplus_{\gotp}K_{\gotp}/x_{\gotp}\gota_{\gotp}.\]
We then define the action of $\mathbb{I}_{K}$ on $K/\gota$ so that the following diagram commutes
\begin{diagram}
K/\gota&\rTo^{x}&K/x\gota   \\
\dTo^{\simeq}& &\dTo_{\simeq} \\  \bigoplus_{\gotp}K_{\gotp}/\gota_{\gotp}&\rTo^{x}&\bigoplus_{\gotp}K_{\gotp}/x_{\gotp}\gota_{\gotp}  \\
\end{diagram}
where the bottom horizontal arrow is defined
\[(t_{\gotp})\mapsto (x_{\gotp}t_{\gotp}).\]

\begin{MainThm}
Let $\mathbb{D}=(\C_{\infty},\uprho)$ be a rank $\text{\rm 1}$ Drinfeld $A$-module. Let $\upsigma\in {\rm Aut}(\C_{\infty}/K)$ and
 let $s\in \mathbb{I}_{K}$ be such that:\[[s,K]=\upsigma|_{K_{A}^{\rm ab}}.\]Let $\gota$ be a fractional ideal of $A$. Given an analytic isomorphism of $A$-modules
\[f:\C_{\infty}/\gota\longrightarrow \mathbb{D},\]
there exists an analytic isomorphism of $A$-modules
\[f':\C_{\infty}/s^{-1}\gota\longrightarrow \mathbb{D}^{\upsigma},\]
whose restriction to $K_{\infty}$ is unique, 
such that the following diagram commutes:\[\begin{diagram}K/\gota&\rTo^{s^{-1}}&K/s^{-1}\gota\\\dTo_{f}& &\dTo_{f'}\\\mathbb{D}&\rTo^{\upsigma}&\mathbb{D}^{\upsigma}\end{diagram}.\]
\end{MainThm}

\begin{note} The statement and proof of the Main Theorem for rank 1 Drinfeld $A$-modules differs from accounts of its elliptic curve analog (c.f.\ \cite{Shi}, \cite{La}, \cite{Si}) in several respects.

\vspace{3mm}

\noindent i. In \cite{Si}, the maps $f$, $f'$ are only assumed to be analytical group isomorphisms of elliptic curves, and not isomorphisms of modules over (an order in) the ring of integers $\mathcal{O}_{K}$
of the complex quadratic extension $K/\Q$. The very definition of Drinfeld module includes the action of ``complex multiplication'' i.e.\ the $A$-action and so in this category, analytic isomorphism
means in particular an isomorphism of $A$-modules.  

\vspace{3mm}

\noindent ii.  On the other hand, note that the version that appears on page 117 of \cite{Shi} is stated in terms of {\it normalized elliptic curves} $(E,\uptheta )$, where $\uptheta:K
\rightarrow {\rm End}_{\Q}(E)$ is a fixed isomorphism.  The use of normalized elliptic curves has the effect that the analog of $f$ is essentially a map of modules over endomorphism rings.
\end{note}
\begin{proof}  

We begin with a standard observation: if $(\mathbb{D}_{1},f_{1})$ is another rank 1 Drinfeld module isomorphic to $\mathbb{D}$, with $f_{1}:\C_{\infty}/\mathfrak{a}\rightarrow \D_{1}$ an analytic isomorphism, then if the Theorem holds for $(\mathbb{D}_{1},f_{1})$, it holds also for $(\mathbb{D},f)$. 
The proof of this is straight-forward, formally identical to that appearing, for example, on pages 160-161 of \cite{Si}.
This allows us to assume that 
\begin{enumerate}
\item[1.] $\D$ is a Hayes module i.e.\ it is sign normalized and its coefficients belong to $H_{A}^{1}$ = the narrow Hilbert class field of $A$.
\item[2.] $\gota$ is an integral ideal of $A$.
\end{enumerate}

Fix $\mathfrak{m}\subset A$ an ideal. The first step will be to show the existence of a commutative diagram 
\begin{align}\label{1step} \begin{diagram}\gotm^{-1}\gota/\gota&\rTo^{s^{-1}}&\gotm^{-1}s^{-1}\gota/s^{-1}\gota\\\dTo_{f}& &\dTo_{f'_{\gotm}}\\\mathbb{D}&\rTo^{\upsigma}&\mathbb{D}^{\upsigma}\end{diagram}
\end{align}
where $f'_{\mathfrak{m}}$ is the restriction of an analytical isomorphism $\C_{\infty}/s^{-1}\mathfrak{a}\rightarrow \D^{\upsigma}$. 
Let $L\subset K^{\rm ab}_{A}$ be a finite Galois extension of $K$ containing
 \begin{itemize}
\item the narrow ray class field $ H^{1}_{A}$: so that $\D$ is defined over $L$.
\item the torsion submodule
 $\D [\mathfrak{m}]$. 
 \item the ray class field\footnote{Hayes' Theorem \cite{Hayes} says that the narrow ray class field $K^{\mathfrak{m},1}_{A}\supset K^{\mathfrak{m}}_{A}$
is generated over $H^{1}_{A}$ by $\D[\mathfrak{m}]$, so in theory we do not need to assume $K^{\mathfrak{m}}_{A}\subset L$: however we shall not assume Hayes' result here.}  $K^{\mathfrak{m}}_{A}$.
 \end{itemize}
Let $\mathcal{O}_{L}\supset A$ be the integral closure of $A$ in $L$ and choose a prime $\mathfrak{P}\subset \mathcal{O}_{L}$ with the following properties: if $\mathfrak{p}=\mathfrak{P}\cap A$, then
   \begin{itemize}
 \item[\ding{192}] $L$ is unramified at $\mathfrak{p}$.
 \item[\ding{193}] $\upsigma|_{L} =[s,K]|_{L}=\upsigma_{\mathfrak{p}}=(\mathfrak{p}, L/K)$.
 \item[\ding{194}] $\mathfrak{p}\not|\mathfrak{m}$.
  \item[\ding{195}]  $\D$ has good reduction at $\mathfrak{P}$.
 \end{itemize}
 
 The existence of such a $\mathfrak{P}$ is guaranteed by the Chebotarev Density Theorem.

Let $\uppi\in \mathbb{I}_{K}$ be such that it has a uniformizer at the $\gotp$ component and is $1$ everywhere else. Then by (\ref{FrobNote})
\[ [\uppi,K]=(\gotp, L/K),\] 
hence $[s\uppi^{-1},K]$ acts trivially on $K_{A}^{\gotm}\subset L$. Therefore, by the idelic characterization of $K_{A}^{\gotm}$ (see (\ref{RayClassChar})),\[s\uppi^{-1}=\upalpha u\]
where $\upalpha\in K^{\times}$ and $u\in U^{\mathfrak{m}}_{A}$ (for all $\mathfrak{q}$ prime, $u_{\mathfrak{q}}\in 1+\mathfrak{m}A_{\gotq}$). 

Consider the isogeny
\[\uprho_{\gotp}: \D\longrightarrow \gotp\ast \D=\D^{\upsigma}\]
where the equality on the right hand side comes from (\ref{HayesThm}). 
We recall by {\it Note} \ref{RedFrob} that its reduction mod $\mathfrak{P}$ satisfies $\widetilde{\uprho_{\gotp}}=\tau^{\deg(\gotp)}$.  The key observation in the proof of the Main Theorem is that we can replace the {\it discontinuous
automorphism} $\upsigma$ by the {\it analytic endomorphism} $\uprho_{\mathfrak{p}}$ if we restrict to $\D[\mathfrak{m}]$: that is, 
we claim that the following diagram commutes
\begin{equation}\label{diagram1}\begin{diagram}\D[\gotm]&\rTo^{\uprho_{\gotp}}&\D^{\upsigma}[\gotm]\\\dInto& &\dInto\\\mathbb{D}&\rTo^{\upsigma}&\mathbb{D}^{\upsigma}\end{diagram},\end{equation}
where the vertical arrows are the inclusions.
Indeed, for every torsion point $t\in \D[\gotm]$, we have the equalities mod $\mathfrak{P}$
\[\widetilde{\uprho_{\gotp}(t)}=\widetilde{\uprho_{\gotp}}(\, \widetilde{t}\,)= \widetilde{t}^{\;\widetilde{\upsigma}}= \widetilde{\;t^{\upsigma}}\] 
since $\widetilde{\uprho_{\gotp}}=\tau^{\deg(\gotp)}=\widetilde{\upsigma}$. 
As $\mathfrak{P}\not|\mathfrak{m}$, by Lemma \ref{injtor}, the reduction map 
\[ \D^{\upsigma}[\mathfrak{m}]\rightarrow\widetilde{ \D^{\upsigma}}[\mathfrak{m}]\]
is injective, hence
\[  \uprho_{\gotp}(t) = t^{\upsigma}. \]

We now fix\[f'':\C_{\infty}/\gotp^{-1}\gota\longrightarrow \D^{\upsigma}\]an analytic isomorphism such that the following diagram commutes:
\[\begin{diagram}\C_{\infty}/\gota&\rTo^{\rm can}&\C_{\infty}/\gotp^{-1}\gota\\\dTo^{f}& &\dTo_{f''}\\\mathbb{D}&\rTo_{\uprho_{\gotp}}&\mathbb{D}^{\upsigma}.\end{diagram}\]
where ${\rm can}$ is the canonical inclusion $x + \gota\mapsto x+\gotp^{-1}\gota$.
The existence of $f''$ is a consequence of the fact that the kernel of the additive homomorphism $\uprho_{\mathfrak{p}}$ is $\D [\mathfrak{p}]$, whose pre-image 
by the analytical $A$-module isomorphism $f$ is exactly $\mathfrak{p}^{-1} \mathfrak{a}$.

Using the decomposition $s=\upalpha\uppi u$, 
we have \[(s)=(\upalpha)(\uppi)=(\upalpha)\gotp.\]Therefore $s^{-1}\gota=(s^{-1})\gota=\upalpha^{-1}\gotp^{-1}\gota$, and 
 multiplication by $\upalpha^{-1}$ gives the following 
isomorphism:\[\upalpha^{-1}:\C_{\infty}/\gotp^{-1}\gota\longrightarrow \C_{\infty}/s^{-1}\gota.\]We can then
form the following 
diagram:\begin{equation}\label{diagram2}\begin{diagram}\C_{\infty}/\gota&\rTo^{\rm can}&\C_{\infty}/\gotp^{-1}\gota&\rTo^{\upalpha^{-1}}&\C_{\infty}/s^{-1}\gota\\\dTo^{f}& &\dTo_{f''}& 
&\dTo_{f'}\\\mathbb{D}&\rTo_{\uprho_{\gotp}}&\mathbb{D}^{\upsigma}&\rTo_{\rm id}&\mathbb{D}^{\upsigma}\end{diagram}\end{equation}where $f'$ is the unique analytic isomorphism making the diagram  commute. 
\begin{clai}  For all $t\in \gotm^{-1}\gota/\gota$, 
\[f(t)^{\upsigma}=f'(s^{-1}t).\]
\end{clai}

\vspace{3mm}

\noindent {\it Proof of the Claim}:
By (\ref{diagram1}) we know that $f(t)^{\upsigma}=\uprho_{\gotp}(f(t))$ for every $t\in \gotm^{-1}\gota/\gota$. 
Using the commutativity of (\ref{diagram2}) the statement reduces to proving that \[f'(\upalpha^{-1}t)=f'(s^{-1}t).\]
This is equivalent to showing that
 \[\upalpha^{-1}t-s^{-1}t\in s^{-1}\gota\]for all $t\in \gotm^{-1}\gota$. Or equivalently,
  \[\upalpha^{-1}t-s_{\gotq}^{-1}t\in s_{\gotq}^{-1}\gota_{\gotq}\]for all $t\in \gotm^{-1}\gota_{\gotq}$ and for all $\gotq$ prime ideals in $A$. 
  Since $s_{\gotq}=\upalpha\uppi_{\gotq}u_{\gotq}$, we are reduced to showing that\[\uppi_{\gotq}u_{\gotq}t-t\in \gota_{\gotq}, \] i.e.\
   \[(\uppi_{\gotq}u_{\gotq}-1)\gota_{\gotq}\subset \gotm\gota_{\gotq}.\]
   Since $u_{\gotq}\in A_{\gotq}^{\times}$ and satisfies $u_{\gotq}\equiv 1$ mod $\mathfrak{m}A_{\gotq}$ we must show that:\[(\uppi_{\gotq}-1)\gota_{\gotq}\subset \gotm\gota_{\gotq}.\]If $\gotq\neq \gotp$, the statement follows since $\uppi_{\gotq}=1$. If $\gotq=\gotp$ we have instead that:\[(\uppi_{\gotp}-1)\gota_{\gotp}=\gota_{\gotp},\]
since $\uppi_{\mathfrak{p}}-1$ is a unit.
By hypotheses, $\gotp\not| \gotm,$ and therefore $\gotm A_{\gotp}=A_{\gotp}$. In particular:\[\gotm\gota_{\gotp}=\gota_{\gotp}.\]This proves the claim. 

\vspace{3mm}

The choice $f'_{\mathfrak{m}} = f'$ makes the diagram (\ref{1step}) commute.
Since $K/\gota=\bigcup_{\gotm}\gotm^{-1}\gota/\gota$, it is enough to
show that these diagrams are compatible and so fit together to produce the diagram appearing in the statement of the Main Theorem.  
Let $\gotn\subset \gotm$; note that \[ \xi:=f'_{\gotn}\circ {f'_{\gotm}}^{-1}\in {\rm Aut}(\mathbb{D}^{\upsigma})=A^{\times}=\mathbb{F}^{\times}_{q}.\] It will be enough to show that \[f'_{\gotn}|_{\gotm^{-1}s^{-1}\gota/s^{-1}\gota}=f'_{\gotm}.\] This follows, since \[\xi f'_{\gotm}(s^{-1}t)=f'_{\gotn}(s^{-1}t)=f(t)^{\upsigma}=f'_{\gotm}(s^{-1}t), \] so that in particular, $\xi$ acts as the identity on $\mathbb{D}^{\upsigma}[\gotm]$ and thus $\xi=1$. 
\end{proof}

\section{Modular Invariant of Rank 1 Drinfeld Modules}\label{ModInvSec}

We establish notation for the discussion which follows, see \S 5 of \cite{Th}. 
Recall that $K_{\infty}$ = the completion of $K$ at $\infty$ may be identified with the Laurent series field
\[ K_{\infty}= \F_{\infty}((u_{\infty}))\]
where $u_{\infty}$ is a uniformizer at $\infty$.    We also choose a sign function
\[ {\rm sgn}:K_{\infty}^{\times}\longrightarrow \F^{\times}_{\infty}\]
so that if $x=c_{N}u_{\infty}^{N} + $ lower, then $\text{sgn}(x)=c_{N}$.   
Let $S$ be a set of representatives of $\F^{\times}_{\infty}/\F^{\times}_{q}$, where we assume $1\in S$.
We say that $f\in K_{\infty}$ is positive if ${\rm sgn}(f)\in S$.  Note then that monic (sgn one) elements are positive and that
any element $f\in K$ has the property that $cf$ is positive for some $c\in\F^{\times}_{q}$.   In addition, we note that by our choice of sign function, adding to a positive element an element of lower degree produces another positive element.

We denote by 
\[ \mathfrak{a}^{+} = \{x\in\mathfrak{a}|\; {\rm sgn}(x)\in S\}. \]
Notice that we have a disjoint union decomposition
\[  \mathfrak{a}^{+} = \bigsqcup_{s\in S} \mathfrak{a}_{s}^{+} ,\quad \mathfrak{a}_{s}^{+} = \{ x\in\mathfrak{a}|\; {\rm sgn}(x)=s\}.\]

Although $S$ need not be a group, for any fixed $t\in S$ and for all $s\in S$, there exists $c_{s}\in \F_{q}$ such that
\[ c_{s} st\in S. \]
We also introduce the notation
\[  \mathfrak{a}_{s} =\F_{q}^{\times}\mathfrak{a}_{s}^{+} \]
which depends only on the class of $s$ in $\F^{\times}_{\infty}/\F_{q}^{\times}$: this allows us to write, abusively, $\mathfrak{a}_{st}$ and
we have
\[    \mathfrak{a}_{s}  \mathfrak{a}_{t} \subset    \mathfrak{a}_{st} .\]

We give a definition of $j$-invariant for rank 1 Drinfeld modules, which was first introduced in \cite{GD3} for quadratic real extensions of $\F_{q}(T)$ and was 
applied to the study of the quantum $j$-invariant.  As before, the formula is based on E.-U. Gekeler's $j$ invariant for rank $2$ complex multiplication Drinfeld modules \cite{Gekeler}. 
Let $\mathbb{D}\simeq \mathbb{C}_{\infty}/\gota$, where $\gota$ is a fractional $A$-ideal. 
We define:\[\zeta^{\gota}(n):=\sum_{x\in \gota^{+}}x^{-n}\]for any $n\in \mathbb{N}$: notice that this expression is only non-0 for elements of the form $n=m(q-1)$ (which are usually called ``even''),
and in what follows we will therefore assume $n$ is of this form.
We call\[J(\gota):=\frac{\zeta^{\gota}(q^{2}-1)}{\zeta^{\gota}(q-1)^{q+1}} \] ands then define 
\begin{align}\label{defnofj} j(\gota):=\frac{1}{\frac{1}{T^{q}-T}-\frac{T^{q^{2}}-T}{(T^{q}-T)^{q+1}}J(\gota)}.\end{align}
Notice that if $\upalpha\in K^{\times}$ has sgn one then
\[  ((\upalpha) \mathfrak{a})^{+}=\upalpha\mathfrak{a}^{+}.\]
This implies that 
\[  j((\upalpha)\mathfrak{a})=  j(\mathfrak{a})\] and so we obtain
a well-defined function on the narrow class group
\[  j=j_{A}:{\sf Cl}_{1}(A)\longrightarrow  \mathbb{C}_{\infty}.\]
In fact, it is a class invariant:

\begin{prop}\label{classinv} $j$ induces a well-defined function 
\[  j:{\sf Cl}(A)\longrightarrow  \mathbb{C}_{\infty}.\]
\end{prop}

\begin{proof} Define for each $s\in S$
\[  \upzeta^{\mathfrak{a}}_{s}(n) = \sum_{x\in\mathfrak{a}^{+}_{s}} x^{-n}\]
so that 
\[  \upzeta^{\mathfrak{a}}(n) = \sum_{s\in S}  \upzeta^{\mathfrak{a}}_{s}(n).\] 
We must show that for any $\upalpha\in K^{\times}$, $j(\upalpha\mathfrak{a})=j(\mathfrak{a})$;
without loss of generality, we may assume ${\rm sgn}(\upalpha)=t\in S$.
Then 
\[  (\upalpha \mathfrak{a})_{st}^{+} = c_{s} \upalpha \mathfrak{a}_{s}^{+} \]
and
\[  \upzeta^{\upalpha\mathfrak{a}}(n)  =\upalpha^{-n} \sum_{s\in S} c^{-n}_{s} \upzeta^{\mathfrak{a}}_{s}(n) . \]
Since $n$ is even, $c^{-n}_{s}=1$, so this gives $ \upzeta^{\upalpha\mathfrak{a}}(n)  = \upalpha^{-n} \upzeta^{\mathfrak{a}}(n) $, and by the formula for $j$, we are done.
\end{proof}

\begin{note}\label{signindepzeta} The proof of Proposition \ref{classinv} shows that $\upzeta^{\mathfrak{a}}(n)$ is independent of the choice of signs $S$: that is, if we replace $s\in S$ by $cs$, $c\in\F^{\times}_{q}$, then
\[ \upzeta^{\mathfrak{a}}_{cs} (n) = c^{-n} \upzeta^{\mathfrak{a}}_{s} (n) =\upzeta^{\mathfrak{a}}_{s} (n) .\]
\end{note}


It will be useful to have the following explicit description of $A$:
\begin{align}\label{Abasis} A =\F_{q}[f_{1},f_{2},\dots ,f_{N}]   = \langle 1, f_{1},f_{2},\dots , f_{N}, f_{N+1},\dots \rangle_{\F_{q}}\end{align}
where the presentation on the far right is that of an $\F_{q}$ vector space, with the additional vector space generators $f_{N+1},\dots$ complementing the ring generators $f_{1},\dots , f_{N}$.  We assume
that $0<\deg (f_{1})<\deg (f_{2})<\cdots $ and by the Weiestrass theory (see for example \cite{HKT}), eventually $\deg (f_{i+1})=\deg (f_{i})+1$.  

The remaining part of this section is devoted to proving results which will be used to show that $j$ is a complete invariant: in other words, we will show that $j$ is injective as a function of ${\sf Cl}(A)$,
see Corollary \ref{cor} of the next section.   This material was presented in 
 \cite{GD3} in the special case when $K$ is quadratic and real.  The arguments we present here are a modification of those appearing in \cite{GD3} and apply to the general case, taking
 into account the notion of positivity when $d_{\infty}>1$.

Fix $\mathfrak{a}\subset A$ a non-principal ideal. We begin by choosing a convenient representative of the class $[\mathfrak{a}] \in{\sf Cl}(A)$ as well as specifying for the representative
an $\F_{q}$-vector space basis.  Write $\mathfrak{a}=(g,h)$ and suppose that $\deg (g)<\deg (h)$.  Without loss of generality, we may assume
 $g$ has the smallest degree of all positive non-zero
elements\footnote{Since $A$ is Dedekind, for any non-zero $x\in\mathfrak{a}$, there exists $y\in\mathfrak{a}$ with $\mathfrak{a}=(x,y)$.
See for example Theorem 17 of \cite{Mar}, page 61.} of $\mathfrak{a}$.  Since $\mathfrak{a}$ is not principal, $h/g$ does not belong to $A$.
 Consider an $\F_{q}$-vector space basis of $\mathfrak{a}$ \begin{align}\label{vecbasis}\{ a_{0}=g,a_{1},\dots \},\end{align} 
 which includes $a_{i_{0}}=h$ for some $i_{0}$ and includes the elements $gf_{1}, gf_{2},\dots$, where the $f_{i}$ are as in (\ref{Abasis}).  
The map $a_{i}\mapsto \deg (a_{i})$ is injective and we will order elements so that $\deg (a_{i})<\deg (a_{j})$ for $i<j$.  

We will also assume that
the $a_{i}$ are positive with respect to the sign function: that is, ${\rm sgn}(a_{i})\in S$.  Note that any element $\upalpha$ having degree equal to $\deg (\upalpha_{i})$ satisfies
${\rm sgn}(\upalpha )\equiv {\rm sgn}(\upalpha_{i})\mod \F^{\times}_{q}$: indeed, $\upalpha =c\upalpha_{i}+$ lower order terms, and since our chosen sign function is unaffected by adding lower order terms, ${\rm sgn}(\upalpha ) = c\cdot {\rm sgn}(\upalpha_{i})$.

Consider the fractional ideal 
\begin{align}\label{astarlist} \mathfrak{a}^{\star}:=g^{-1}\mathfrak{a} = \langle \upalpha_{0}=1,\upalpha_{1}, \upalpha_{2},\dots ,\upalpha_{n},f_{1}, f_{2},\dots  \rangle_{\F_{q}},\quad \upalpha_{i}:=a_{i}/g.\end{align}
  Since $\mathfrak{a}$ is not principal, we have 
$ \mathfrak{a}^{\star}\supsetneq A=(1)$.   

\begin{lemm}\label{basislemma} $|\upalpha_{1}|\not=|f_{1}|$.
\end{lemm}

\begin{proof}  Suppose $|\upalpha_{1}|=|f_{1}|$.  Then $\upalpha_{1}=a_{1}/g$ has degree which appears in the list of degrees realized by elements of $A$. Thus there exists $a\in A$ with 
\[ \deg (a_{1}')<\deg (a_{1}),\quad a_{1}' := a_{1}+ag \in\mathfrak{a} .\] By the shape of the basis of $\mathfrak{a}$, $\deg (a_{1}')=\deg (g)$, but this implies that 
\[    a_{1}'= cg \]
for some $c\in \F^{\times}_{q}$ or $a_{1}= (c-a)g$.  Therefore, $\upalpha_{1} = c-a\in A$ and since $\deg (\upalpha_{1})=\deg (f_{1})$ we have $\upalpha_{1}= c_{1}f_{1}+d_{1}$ for $c_{1}\in\F_{q}^{\times}$ and $d_{1}\in\F_{q}$.  But this contradicts
the fact that the elements displayed in (\ref{astarlist}) are independent.
\end{proof}


In what follows we will work with $\mathfrak{a}^{\star}=g^{-1}\mathfrak{a}$, renaming it $\mathfrak{a}$ and 
fixing the basis satisfying the conditions described above.  That is,
\[\mathfrak{a}=\langle 1,\upalpha_{1},\dots \rangle_{\F_{q}},\]
where $\upalpha_{i}=a_{i}/g$, as specified in the paragraph before Lemma \ref{basislemma}.  Notice that since ${\rm sgn}(g)=t$ is not necessarily 1, 
the basis  $\{ 1,\upalpha_{1},\dots \}$ is no longer positive with respect to $S$ but rather with respect to the new choice of sign representatives $t^{-1}S$. However by {\it Note} \ref{signindepzeta}
the zeta function does not depend on the choice of $S$, so this will not effect any calculations which we make with them.


For $n\in\N$, consider the zeta function
\[ \upzeta^{\mathfrak{a}}(n(q-1)) =\sum_{x\in \mathfrak{a}^{+}} x^{n(1-q)} = 
1+\sum_{i=1}^{\infty} \Upomega_{i}^{\mathfrak{a}}(n(q-1))\]
where 
\begin{align*} \Upomega_{i}^{\mathfrak{a}}(n(q-1)) & =  \sum_{c_{0},\dots ,c_{i-1}\in \F_{q}} \left(c_{0}+c_{1}\upalpha_{1}+\dots + c_{i-1}\upalpha_{i-1} + \upalpha_{i}\right)^{n(1-q)} \\
&:= \sum_{\vec{c}\in\F_{q}^{i}} \left(\vec{c}\cdot\vec{\upalpha}_{i-1} +\upalpha_{i}\right)^{n(1-q)}
  \end{align*}
  and where $\vec{\upalpha}_{i-1}=(\upalpha_{0}=1,\upalpha_{1},\dots ,\upalpha_{i-1})$.  Note that the coefficient of $\upalpha_{i}$ in the above must be 1 in order to maintain positivity 
  (recall from the discussion above that there is a unique possible sign in each degree, modulo $\F^{\times}_{q}$).  Thus $\Upomega_{i}^{\mathfrak{a}}$ is the subsum of the zeta function consisting of the sums
  of positive elements having degree $=\deg(\upalpha_{i})$.   

\begin{lemm}\label{omegalemma}  Let $\mathfrak{a}$, $ \Upomega_{i}^{\mathfrak{a}}(q^{n}-1)$ 
be as above.  
\begin{enumerate}
\item When $i=1$, 
\[  \Upomega_{1}^{\mathfrak{a}}(q^{n}-1)  =\frac{\upalpha_{1}^{q^{n}}-\upalpha_{1}}{\prod_{c\in\F_{q}}(c+\upalpha_{1}^{q^{n}})}\quad\text{and}\quad 
 \left| \Upomega_{1}^{\mathfrak{a}}(q^{n}-1)\right| = |\upalpha_{1}|^{q^{n}(1-q)}.\]
\item  For all $i\geq 1$,
\begin{align*} \left|  \Upomega_{i}^{\mathfrak{a}}(q^{n}-1)\right| &  
\leq |\upalpha_{i}|^{q^{n}(1-q)}.
\end{align*}
\end{enumerate}
\end{lemm}

\begin{proof} 
Write $\upalpha=\upalpha_{1}$.  Then 
\begin{align}\label{lemmafraction} \Upomega_{1}^{\mathfrak{a}}(q^{n}-1)=
\sum_{c}\frac{c+\upalpha}{c+\upalpha^{q^{n}}}=
 \frac{\sum_{c} (c+\upalpha)\prod_{d\not=c}(d+\upalpha^{q^{n}})}{\prod_{c}(c+\upalpha^{q^{n}})} .
\end{align}
Denote by $s_{i}(c)$ the $i$th elementary symmetric function on $\F_{q}-\{ c\}$.  Thus, $s_{0}(c)=1$, $s_{1}(c)=\sum_{d\not=c}d$, etc. Then the numerator of (\ref{lemmafraction}) may be written as 
\begin{align*}
\sum_{j=0}^{q-1}\bigg(\sum_{c} (c+\upalpha)s_{q-1-j}(c)\bigg)\upalpha^{jq^{n}} .\\
\end{align*}
First note that there is no constant term, and the coefficient of $\upalpha$ is $\prod_{d\not=0}d=-1$.
Now
\[  \sum_{c}cs_{q-2}(c)= \sum_{c\not=0}cs_{q-2}(c)=\sum_{c\not=0} c\prod_{d\not=0,c}d =\sum_{c\not=0}c\cdot (- c^{-1})=(q-1)(-1)=1,\]
which is the coefficient of $\upalpha^{q^{n}}$.
Moreover, $\sum_{c}s_{q-2}(c)=  s_{q-2}(0)  -\sum_{c\not=0}c^{-1}=0$, so the $\upalpha^{q^{n}+1}$ term vanishes.  
For $i<q-2$, we claim that
\begin{align*}
\sum_{c} cs_{i}(c)=0= \sum_{c} s_{i}(c). \end{align*}
When $i=0$,  $s_{0}(c)=1$ for all $c$, 
the terms $\upalpha^{q^{n}(q-1)}$, $\upalpha^{q^{n}(q-1)+1}$ have coefficients $\sum_{c}c=\sum_{c}1=0$ and so vanish.
When $i=1$, we have $q>3$,
so $s_{1}(c)=-c$ and  \[  \sum_{c}s_{1}(c)=-\sum_{c}c=0=-\sum_{c}c^{2} =\sum_{c}cs_{1}(c)\]
since the sums occurring above are power sums over $\F_{q}$ of exponent $1,2<q-1$.
For general $i<q-2$, we have $q>i+2$ and
\[ \sum_{c} s_{i}(c) = \sum_{c} P(c)\]
where $P(X)$ is a polynomial over $\F_{q}$ of degree $i<q-2$.   Hence  $\sum_{c} P(c)= \sum_{c} cP(c)=0$, since again, these are sums of powers of $c$ of exponent less than $q-1$.
Thus 
the numerator is $ \upalpha^{q^{n}}-\upalpha $
and the absolute value claim follows immediately.
When $i>1$, for each $\vec{c}$, let $\vec{c}_{+}=(c_{1},\dots, c_{i-1})$ and write 
\[ \upalpha_{\vec{c}_{+}} = c_{1}\upalpha_{1} + \cdots + c_{i-1}\upalpha_{i-1} +\upalpha_{i}.\]
Note trivially that $|\upalpha_{\vec{c}_{+}}|=|\upalpha_{i}|$.  Then by part (1) of this Lemma,
\begin{align*}
 \left|  \Upomega_{i}^{\mathfrak{a}}(q^{n}-1)\right| = \left| \sum_{\vec{c}_{+}}\sum_{c} (c+\upalpha_{\vec{c}_{+}})^{1-q^{n}}\right|
 \leq \max \{ |\upalpha_{\vec{c}_{+}}|^{q^{n}(1-q)} \} =  |\upalpha_{i}|^{q^{n}(1-q)}.
\end{align*}
\end{proof}

In what follows, we write 
\[ \hat{\upzeta}^{\mathfrak{a}}(n(q-1)) = \upzeta^{\mathfrak{a}}(n(q-1))-1 .\]
By Lemma \ref{omegalemma}, since $|\upalpha_{i}|>|\upalpha_{1}|>1$ for $i>1$, we have immediately

\begin{coro}\label{boundonzetaa}  Let $\mathfrak{a}$ be as above.  Then for all $n\in\N$, 
\[  \left|\hat{\upzeta}^{\mathfrak{a}}(q^{n}-1)\right|  =\left|  \Upomega_{1}^{\mathfrak{a}}(q^{n}-1)\right|
=|\upalpha_{1}|^{q^{n}(1-q)} <1.
\]
\end{coro}

\begin{theo}\label{jadifffrom1}  For all $\mathfrak{a}\subset A$ non principal, $j(\mathfrak{a})\not=j((1))$. 
\end{theo}

\begin{proof}  
It will be enough to prove
the Theorem with $\mathfrak{a}$ replaced by
the fractional ideal
$g^{-1}\mathfrak{a}$ studied above.
For any ideal $\mathfrak{b}$ we denote $\tilde{J}(\mathfrak{b})= \upzeta^{\mathfrak{b}}(q^{2}-1)/\upzeta^{\mathfrak{b}}(q-1)^{q+1}$.
It will suffice to show that $\tilde{J}(\mathfrak{a})\not= \tilde{J}((1))$ i.e. that the numerator of
\[  \tilde{J}(\mathfrak{a})- \tilde{J}((1)) =\frac{\upzeta^{\mathfrak{a}}(q^{2}-1)\cdot
\upzeta^{(1)}(q-1)^{q+1}-\upzeta^{\mathfrak{a}}(q-1)^{q+1}\cdot
\upzeta^{(1)}(q^{2}-1)}{ \left(\upzeta^{\mathfrak{a}}(q-1)\cdot \upzeta^{(1)}(q-1)\right)^{q+1} }\]
does not vanish.  This numerator can be written
\begin{align}\label{numerator} (\hat{\upzeta}^{\mathfrak{a}}(q^{2}-1)+1)(\hat{\upzeta}^{(1)}(q-1)^{q}+1)(\hat{\upzeta}^{(1)}(q-1)+1)- &  \\
 (\hat{\upzeta}^{(1)}(q^{2}-1)+1)(\hat{\upzeta}^{\mathfrak{a}}(q-1)^{q}+1)(\hat{\upzeta}^{\mathfrak{a}}(q-1)+1). \nonumber 
 \end{align}
 Developing the products, by Corollary \ref{boundonzetaa} we see that (\ref{numerator}) can be written
 \[ -\hat{\upzeta}^{\mathfrak{a}}(q-1)+\hat{\upzeta}^{(1)}(q-1)   + 
  \hat{\upzeta}^{\mathfrak{a}}(q^{2}-1)-\hat{\upzeta}^{(1)}(q^{2}-1)+\text{lower}. \]
  By Lemma \ref{basislemma}, $|\upalpha_{1}|\not= |f_{1}|$: since the terms appearing above are symmetric, mod $\pm 1$, in $\mathfrak{a}$ and $(1)$, we may assume, without loss of generality, that $|\upalpha_{1}|<|f_{1}|$.
 Therefore,  by Lemma \ref{omegalemma} and Corollary \ref{boundonzetaa}, the absolute value of (\ref{numerator}) is $|\upalpha_{1}|^{q(1-q)}>0$ since
\begin{align*}  \left|\hat{\upzeta}^{\mathfrak{a}}(q-1) \right|=|\upalpha_{1}|^{q(1-q)} & >\max \left\{  |f_{1}|^{q(1-q)} ,|\upalpha_{1}|^{q^{2}(1-q)}  ,|f_{1}|^{q^{2}(1-q)}  \right\} \\
& = 
\max \left\{\left|\hat{\upzeta}^{(1)}(q-1) \right|, 
\left|\hat{\upzeta}^{\mathfrak{a}}(q^{2}-1)\right|,\left|\hat{\upzeta}^{(1)}(q^{2}-1) \right| 
\right\} \end{align*}
  and we are done.
 \end{proof}

\section{Class Field Generation}\label{CFGen}

We now use the Main Theorem to prove the following key result about the $j$-invariant defined in the previous section.

\begin{theo}\label{t3}  
Let $\mathfrak{a}$ be a fractional ideal of $A$ and let $s\in \I_{K}$ be a $K$-id\`{e}le.  Then $j(\mathfrak{a})\in K^{\rm ab}_{A}$
and 
\[  j(s^{-1}\mathfrak{a}) = j(\mathfrak{a})^{[s,K]}. \]
\end{theo}

\begin{proof}
We first show that $j(\gota)\in H^{1}_{A}$. 
Let us call
\begin{align}\label{normzeta}\widetilde{\upzeta}^{\gota}(n):=\frac{\upzeta^{\gota}(n)}{{\xi_{\gota}}^{n}}\end{align}
the \textsl{normalized} 
value of $\upzeta^{\gota}(n)$, where $\xi_{\gota}$ is the transcendental element of $\C_{\infty}$ associated to $\mathfrak{a}$, used to produce sign normalization (see \S 12 of \cite{Hayes}).
Because $j(\gota)$ is defined using a homogeneous ratio of values of the 
zeta function $\upzeta^{\gota}$ (see equation (\ref{defnofj})), we may replace these values by their normalized values. 
Then, by the version of  D. Goss' Theorem stated in \cite{Th} (Theorem 5.2.5), the normalized values are all in $H^{1}_{A}$. (N.B. The statement appearing in \cite{Th} is that $\upxi_{\mathfrak{a}}^{-n}\cdot \sum_{0\not= x\in\mathfrak{a}} x^{-n}\in H^{1}_{A}$.
  However, since $\mathfrak{a}\setminus \{ 0\} = \bigsqcup_{c\in\F_{q}^{\times}} c\mathfrak{a}^{+}$,  
  \[  \sum_{0\not= x\in\mathfrak{a}} x^{-n}=\# S\cdot \upzeta^{\mathfrak{a}}(n)=  \upzeta^{\mathfrak{a}}(n) \]
as $\#S=(q^{d_{\infty}}-1)/(q-1) \equiv 1\mod q$.) Let $\upsigma=[s,K]$. 
We claim that\begin{equation} \widetilde{\upzeta}^{\gota}(n)^{\upsigma}=\widetilde{\upzeta}^{s^{-1}\gota}(n).\end{equation}
Let $\mathbb{D}=(\C_{\infty},\uprho)$ be the Hayes module attached to the lattice $\xi_{\gota}\gota$. First note that if $e_{\uprho}(z)=e_{\xi_{\gota}\gota}(z)$ is the associated exponential function, then taking its logarithmic derivative we get
\[\frac{1}{e_{\uprho}(z)}=\sum_{\upalpha\in\mathfrak{a}} \frac{1}{z+\upxi_{\gota}\upalpha}
=
-\sum_{n=0}^{\infty}\sum_{\upalpha\in \gota}\frac{z^{n}}{(\xi_{\gota}\upalpha)^{n+1}}
=- \sum_{n=0}^{\infty}\widetilde{\upzeta}^{\gota}(n+1)z^{n}.\]
Therefore 
\[ e_{\uprho}(z)=\sum_{n=0}^{\infty}c_{n}z^{q^{n}}, \quad c_{n}\in H^{1}_{A}, \]
 where the $c_{n}$ are algebraic 
combinations of the $\widetilde{\upzeta}^{\gota}(n+1)$ of a universal form which is dictated by the formula for the reciprocal of a power series. 
Fix $a\in A$ and write \[\uprho_{a}(\tau)=a+g_{1}\tau+\cdots +g_{d}\tau^{d},\quad g_{1},\dots ,g_{d}\in H_{A}^{1} .\]
Then, the equation
\[e_{\uprho}(az)=\uprho_{a}(e_{\uprho}(z))\]
implies 
that 
\begin{align*} az+c_{1}a^{q}z^{q}+(c_{2}a^{q^{2}})z^{q^{2}}+\cdots & = \uprho_{a}(z+c_{1}z^{q}+\cdots ) \\
& =\uprho_{a}(z)+\uprho_{a}(c_{1}z^{q})+\cdots \\ &=az+(g_{1}+ac_{1})z^{q}+(g_{2}+c_{1}^{q}g_{1}+ac_{2})z^{q^{2}}+\cdots \end{align*}
The shape of the last expression above is again of a universal nature and depends only on the coefficients of $\uprho_{a}$ and the coefficients $e_{\uprho}(z)$.
That is, we have \[c_{1}=\frac{g_{1}}{a^{q}-a}, \quad
c_{2}=\frac{g_{2}+c_{1}^{q}g_{1}}{a^{q^{2}}-a},\quad \dots \]
and the coefficients of $e_{\uprho}(z)$ may be solved for in terms of the coefficients of $\uprho_{a}$ using a universal recursion.  In particular, we have given a formula
for the coefficients of the normalized exponential attached to any Hayes module $(\D,\uprho )$, which only depends on the coefficients of $\uprho_{a}$ for $a\in A$ fixed.
As $\mathbb{D}^{\upsigma}=(\C_{\infty},\uprho^{\upsigma})$ has lattice homothetic to $s^{-1}\gota$ by the Main Theorem, 
it follows that
\[e_{\uprho^{\upsigma}}(z)=   e_{\upxi_{s^{-1}\mathfrak{a}} s^{-1}\mathfrak{a}} (z)   =\sum_{n=0}^{\infty}c_{n}^{\upsigma}z^{n}  \]
where $\upxi_{s^{-1}\mathfrak{a}} $ is the transcendental factor producing the sign normalized Drinfeld module associated to $s^{-1}\mathfrak{a}$.
Therefore for all $n$
\begin{align}  \frac{\upzeta^{s^{-1}\mathfrak{a}}(n)  }{\upxi^{n}_{s^{-1}\mathfrak{a}} } =  \widetilde{\upzeta}^{s^{-1}\gota}(n) = (\widetilde{\upzeta}^{\gota}(n))^{\upsigma}.\end{align}
The statement about the $j$-invariant follows immediately.
\end{proof}

\begin{coro}\label{cor}  The $j$-invariant takes values in $H_{A}\subset H_{A}^{1}$ and the function \[ j:{\sf Cl}(A)\longrightarrow H_{A}\]
is injective.
\end{coro}
\begin{proof}  Recall that ${\sf P}(A)$ denotes the group of principal ideals in $A$ and ${\sf P}_{1}(A)$ the subgroup of principal ideals containing a generator of sgn one.  Then by Theorem \ref{t3}
the action of
\[ \upsigma \in {\rm Gal}(H^{1}_{A}/H_{A})\cong  {\sf P}(A)/{\sf P}_{1}(A) \subset {\sf Cl}_{1}(A)\]
is trivial on $j(\mathfrak{a})$, since $j$ is an invariant of the usual class group ${\sf Cl}(A)$ by Proposition \ref{classinv}.  Thus $j(\mathfrak{a})\in H_{A}$.
If $j(\gota)=j(\gotb)$, by Theorem \ref{t3} there exists $\upsigma\in \Gal(H_{A}/K)$ corresponding to the analytic action of $\gotb^{-1}$, such that
\[j(\gota\gotb^{-1})=j(\gota)^{\upsigma}=j(\gotb)^{\upsigma}=
j(\gotb\gotb^{-1})=
j((1)).\]By Theorem \ref{jadifffrom1} $\mathfrak{a}\mathfrak{b}^{-1}$ is principal and $[\mathfrak{a}]=[\mathfrak{b}]$.  \end{proof}

\begin{theo}
\[H_{A}=K(j(\gota)).\]
\end{theo}
\begin{proof}
By Theorem \ref{t3}, the image $j({\sf Cl}(A))\subset H_{A}$ consists of a Galois orbit and has maximal size by Corollary \ref{cor}. Since $H_{A}/K$ is an abelian field extension, $K(j(\gota))/K$ is a Galois extension contained in $H_{A}$, hence the two fields are equal having the same degree over $K$. 
\end{proof} 




Let $\mathfrak{m}\subset A$ be an ideal and $K^{\mathfrak{m}}_{A}$ the associated ray class field.  See the end of \S\ref{CFT} for this and other related notation. By Class Field Theory, $K_{A}^{\mathfrak{m}}\subset K^{\rm ab}_{A}$ is the fixed field of the group of Artin symbols $[s,K]\in {\rm Aut}(K^{\rm ab}_{A}/K)$, where
 $s\in K^{\times} U^{\mathfrak{m}}_{A}$.
By definition, an analytical isomorphism $f:\C_{\infty}/\mathfrak{a}\rightarrow \D = (\C_{\infty},\uprho )$ of $A$-modules takes the analytical torsion 
\[ {\rm Tor}_{\mathfrak{a}}(\mathfrak{m}):=\mathfrak{m}^{-1}\mathfrak{a}/\mathfrak{a}\] to the algebraic torsion $\D[\mathfrak{m}]$.  

\begin{lemm}\label{torlem} Suppose that $s\in \I_{K}$ for which $s\mathfrak{a}=\mathfrak{a}$.  Then $s$ acts as the identity on ${\rm Tor}_{\mathfrak{a}}(\mathfrak{m})\subset K/\mathfrak{a}$
$\Leftrightarrow $  $s\in U^{\mathfrak{m}}_{A}K^{\times}$.
\end{lemm}

\begin{proof} Suppose $s\in U^{\mathfrak{m}}_{A}K^{\times}$: since $K^{\times}$ is in the kernel of the Artin map, without loss of generality we may assume $s\in  U^{\mathfrak{m}}_{A}$.  Then for all $\mathfrak{p}\subset A$ prime, we have 
\[ (s_{\mathfrak{p}}-1) \mathfrak{m}_{\mathfrak{p}}^{-1}\mathfrak{a}_{\mathfrak{p}} \subset   \mathfrak{a}_{\mathfrak{p}}\]
so $s$ acts trivially on $\mathfrak{m}^{-1}\mathfrak{a}/\mathfrak{a}$.  On the other hand, if $s$ acts as the identity on ${\rm Tor}_{\mathfrak{a}}(\mathfrak{m})$, we have for each $\mathfrak{p}$ the above inclusion mod $K^{\times}$ and
then $ (s_{\mathfrak{p}}-1)\in\mathfrak{m}_{\mathfrak{p}}$ mod $K^{\times}$.
\end{proof}

 Let $\D =(\C_{\infty},\uprho )$ be defined over the minimal field of definition $H_{A}$ and let $e_{\uprho}$ be the exponential inducing an isomorphism
 \[   e_{\uprho}: \C_{\infty}/\Uplambda_{\uprho}\longrightarrow \D.  \]
 Then there exists $\mathfrak{a}\subset A$ an ideal and $\upxi\in \C_{\infty}$ so that $\Uplambda_{\uprho}= \upxi \mathfrak{a}$.  We stress that $\D$ is in general
 {\it not} sign normalized, unless $d_{\infty}=1$.

\begin{theo} Let $\D$ be as in the previous paragraph.  Then
\[  K_{A}^{\mathfrak{m}} = H_{A}(e_{\uprho} (\upxi t)|\; t\in {\rm Tor}_{\mathfrak{a}}(\mathfrak{m}) ) .\]
\end{theo}

\begin{proof} By definition, $K^{\mathfrak{m}}_{A}$ is abelian. Moreover, $K(e_{\uprho} (\upxi t)|\; t\in {\rm Tor}_{\mathfrak{a}}(\mathfrak{m}))$ is abelian, by Theorem 3.1.1 of \cite{Th},
and therefore $H_{A}(e_{\uprho} (\upxi t)|\; t\in {\rm Tor}_{\mathfrak{a}}(\mathfrak{m}) )$ is abelian as well.
  Therefore,
it will be enough to show that the two fields appearing in the statement of the Theorem are the fixed fields of the same subgroup of ${\rm Gal}(K^{\rm ab}_{A}/K)$.  Let $\upsigma = [s, K]$.  Then 
\[  \upsigma|_{K^{\mathfrak{m}}_{A}}\text{ is trivial }\Longleftrightarrow s^{-1} =\upalpha u, \;\; \upalpha\in K^{\times} \text{ and } u\in U_{A}^{\mathfrak{m}}. \]
Suppose first that $\upsigma|_{K^{\mathfrak{m}}_{A}}$ is trivial.  Then $\D^{\upsigma}=\D$ since $\D$ is defined over $H_{A}\subset K^{\mathfrak{m}}_{A}$.  Choose the analytical isomorphism 
$f$ in the Main Theorem to be the composition $\mathfrak{a}\rightarrow \upxi\mathfrak{a}$
with the exponential $e_{\uprho}$.
By the Main Theorem, it follows that $s^{-1}\mathfrak{a}$ also parametrizes $\D$ and therefore is a multiple of $\mathfrak{a}$ by an element of $K^{\times}$.
In particular, we may choose the element $\upalpha$ above so that $s^{-1}\mathfrak{a}=\mathfrak{a}$.  
By Lemma \ref{torlem}, $s$ acts trivially on the analytic torsion ${\rm Tor}_{\mathfrak{a}}(\mathfrak{m})$, and by the Main Theorem, it follows that $\upsigma$
 fixes $H_{A}( e_{\uprho} (\upxi t)|\; t\in {\rm Tor}_{\mathfrak{a}}(\mathfrak{m}) ) $.  
  In the other direction, suppose $\upsigma=[s,K]$ fixes $H_{A}( e_{\uprho} (\upxi t)|\; t\in {\rm Tor}_{\mathfrak{a}}(\mathfrak{m}) ) $.  In particular, $\D^{\upsigma}=\D$,
and after re-choosing $\upalpha$, we may assume $s^{-1}\mathfrak{a}=\mathfrak{a}$. By the Main Theorem, we conclude $s^{-1}$ acts trivially on ${\rm Tor}_{\mathfrak{a}}(\mathfrak{m})$ as well, and by Lemma
\ref{torlem},
$s\in U^{\mathfrak{m}}_{A}K^{\times}$.  Therefore $\upsigma=[s,K]$ is the identity on $K^{\mathfrak{m}}_{A}$.
\end{proof}



\begin{thebibliography}{00}
\bibitem[1]{AngPell} Angl\`{e}s, Bruno \& Pellarin, Federico, Universal Gauss-Thakur sums and L-series.  {\it Invent. Math.} {\bf 200} (2015), no.2, 653--669. 
\bibitem[2]{Auer}Auer, Roland, Ray class fields of global function fields
with many rational places. {\it Acta Arith.} {\bf 95} (2000), no. 2, 97--122.
\bibitem[3]{Cohn}Cohn, P.M., {\it Algebraic Numbers and Algebraic Functions.} Chapman and Hall/CRC, London, 1991.
\bibitem[4]{Drinfeld} Drinfeld, V.G., Elliptic modules, {\it Math. Sbornik} {\bf 94} (1974), 594--627.
\bibitem[5]{GD1}Demangos, L. \& Gendron, T.M., Quantum $j$-Invariant in Positive Characteristic I: Definitions and Convergence. Arch. Math. {\bf 107} (2016), no. 1, 23--35.
\bibitem[6]{GD2}Demangos, L. \& Gendron, T.M., Quantum $j$-Invariant in Positive Characteristic II: Formulas and Values at the Quadratics. Arch. Math. {\bf 107} (2016), no. 2, 159--166.
\bibitem[7]{GD3}Demangos, L. \& Gendron, T.M., Quantum Drinfeld modules I: quantum modular invariant and Hilbert class fields. (2016) arXiv:1607.03027.
\bibitem[8]{GD4} Demangos, L. \& Gendron, T.M., Quantum Drinfeld modules II: quantum exponential and ray class fields. (2017) arXiv:1709.05337. 
\bibitem[9]{Gekeler} Gekeler, Ernst-Ulrich,  Zur Arithmetik von Drinfeld Moduln, Math. Ann. {\bf 262} (1983), 167--182. 
\bibitem[10]{Goss} Goss, David.,  \emph{Basic Structures of Function Field Arithmetic}, Springer-Verlag, 1996.
\bibitem[11]{Hayes2}Hayes, David R., Explicit class field theory in global function fields, in {\it Studies in Algebra and Number Theory},  Adv. in Math. Suppl. Stud. {\bf 6}, Academic Press, 1979, pp. 173--217.
\bibitem[12]{Hayes} Hayes, David R., A brief introduction to Drinfeld modules, in \emph{The arithmetic of function fields} (ed. D. Goss, D. R. Hayes, M. I. Rosen) 
Ohio State U. Mathematical Research Institute Publications {\bf 2}, pp. 313 - 402, Walter De Gruiter, Berlin, 1992.
\bibitem [13]{HKT} Hirschfeld, J.W.P., Korchm\'{a}ros, G., \& Torres, F., {\it Algebraic Curves over a Finite Field}, Princeton Series in Applied Mathematics, Princeton U. Press, Princeton, NJ, 2008.
\bibitem[14]{La} Lang, Serge, {\it Elliptic Functions}, 2nd ed., Graduate Texts in Mathematics {\bf 112}, Springer-Verlag, New York, 1987.
\bibitem [15]{Mar} Marcus, Daniel A., {\it Number Fields}. Universitext, Springer-Verlag, New York, 1995.
\bibitem[16]{Ro} Rosen, Michael, The Hilbert class field in function fields.  {\it Expo. Math} {\bf 5} (1987), 365--378.
\bibitem[17]{Serre} Serre, J.-P., Complex multiplication, in {\it Algebraic Number Theory} (ed. J.W.S. Cassels and A. Fr\"{o}lich), 2nd edition, pp. 293--296, London Mathematical Society, 2010.
\bibitem[18]{Shi} Shimura, Goro, {\it Introduction to the Arithmetic Theory of Automorphic Functions}, Publications of the Mathematical Society of Japan, {\bf 11}. Kan\^{o} Memorial Lectures, {\bf 1}. Princeton University Press, Princeton, NJ, 1971.
\bibitem[19]{Si} Silverman, Joseph H., {\it Advanced Topics in the Arithmetic of Elliptic Curves}, Graduate Texts in Mathematics {\bf 151}, Springer-Verlag, New York, 1994.
\bibitem[20]{Ta} Tate, J.T., Global class field theory, in {\it Algebraic Number Theory} (ed. J.W.S. Cassels and A. Fr\"{o}lich), 2nd ed., pp. 162--203, London Mathematical Society, London, U.K., 2010.
\bibitem[21]{Th} Thakur, Dinesh S., {\it Function Field Arithmetic}, World Scientific, Singapore, 2004.
\end{thebibliography}
\end{document}